\documentclass[12pt]{elsarticle}
\usepackage[letterpaper,top=2cm,bottom=2cm,left=3cm,right=3cm,marginparwidth=1.75cm]{geometry}
\usepackage{amsmath}
\usepackage{amsthm}
\usepackage{amsfonts}
\usepackage{graphicx}
\usepackage{longtable}
\usepackage{multirow}
\usepackage{tikz}
\usepackage{overpic}
\usepackage{wrapfig}
\usepackage{multicol}
\newtheorem{Th}{Theorem}
\newtheorem{Lemma}{Lemma}
\usepackage[colorlinks=true, allcolors=blue]{hyperref}
\usepackage{url}
\usepackage{hyperref}

\usepackage{caption}
\usepackage{subcaption}
\graphicspath{{figures/}}
\makeatletter
\def\ps@pprintTitle{%
 \let\@oddhead\@empty
 \let\@evenhead\@empty
 \def\@oddfoot{Preprint submitted \it\hfill\today}%
 \let\@evenfoot\@oddfoot}
\makeatother

\begin{document}
\journal{}
\begin{frontmatter}
\title{Reliable A Posteriori Error Estimator for a Multi-scale Cancer Invasion Model}
\author[1,2]{Gopika P. B.}
\ead{gopikapb23@iisertvm.ac.in}
\author[1]{Nishant Ranwan}
\ead{nishant21@iisertvm.ac.in}
\author[1,2]{Nagaiah Chamakuri}
\ead{nagaiah.chamakuri@iisertvm.ac.in}
\address[1]{School of Mathematics, IISER Thiruvananthapuram, 695551 Kerala, India.}
\address[2]{Center for High-Performance Computing, IISER Thiruvananthapuram, 695551 Kerala, India}

\begin{abstract}
In this work, we analyze the residual-based a posteriori error estimation of the multi-scale cancer invasion model, which is a system of three non-stationary reaction-diffusion equations. We present the numerical results of a study on a posteriori error control strategies 
for FEM approximations of the model.
In this paper, we derive a residual type error estimator for the cancer invasion model and illustrate its practical performance on a series of computational tests in three-dimensional spaces.
We show that the error estimator is reliable and efficient regarding the model's small perturbation parameters. 
\end{abstract}

\end{frontmatter}
\section{Introduction}

A cell is known to be the fundamental unit of life in all living organisms. For the organisms to operate accurately, these cells must develop and split in a controlled manner according to a specific set of rules. Generally, cells are immobile and can develop, multiply, and kill themselves in a self-regulated mode. Occasionally, cells may expand freely without considering the standard ratio of growth and death. Cells damage host tissue if they develop and reproduce by forgetting the body's needs and constraints.  
In a broad sense, cancer is a condition in which normal cells begin to replicate uncontrollably. It is a  category of disease that has numerous origins and evolves over time and space. Therefore, cancer is a complex phenomenon to foresee. Comprehending the mechanism of cancer progression is critical to detecting and treating this disease. Understanding the process of cancer evolution relies on numerical simulations to a great extent, which entrusts us to see the growth and spreading of the tumor. Cancer models are treated as reaction-diffusion equations in numerous works of literature \cite{anderson2000mathematical}. 

\par
There is ample literature for a posterior analysis of finite element methods for elliptic and parabolic equations, to note a few \cite{Verfuerth13, Ainsworth_book00,estep2000estimating}. 
The robust, reliable, and efficient residual-based error estimate for the singularly perturbed reaction-diffusion problem was derived in \cite{Verfurth_NM98,Ainsworth_CMAME14,Ainsworth_SINUM99, Vejchodsky_IMAJNA06}.
In \cite{Verfuerth13}, reliable and efficient residual-based error estimators are discussed for linear and nonlinear elliptic equations and linear and quasilinear parabolic equations. In \cite{estep2000estimating}, the duality approach is used for a posteriori error estimation for a system of reaction-diffusion equations. To the author's knowledge, theoretical a-posterior error estimates  for the cancer invasion model were not done in the past. Also, previous works focused less on a coupled system of nonlinear parabolic equations to derive the theoretical estimates, where one unknown depends on the gradient of the other.
The basic idea behind our work can be traced back to the method of residual-based error estimation in \cite{Verfuerth13}.

\par
Numerical techniques to solve the tumor invasion model and associated continuum mathematical models are discussed in the following. Solving a mathematical model of angiogenesis by using the finite difference scheme was presented in  \cite{Anderson_BMB98, chaplain2005mathematical}.
In \cite{Daarud_IJNMBE12}, the research delved into the examination of the four-species tumor growth model through the application of a two-dimensional mixed Finite Element Method (FEM). Meanwhile, \cite{Vilanova_IJNMBE13} presented numerical findings for the continuum tumor invasion model, employing the FEM while considering the expansion of capillaries within a two-dimensional spatial domain.
To enhance computational efficiency in practical real-world scenarios, Adaptive Mesh Refinement (AMR) techniques, whether in the temporal or spatial domain, have found favorable utilization, as discussed in references \cite{Cherry_PRL00, Trangenstein_JCP04, Peterson_IJNME07, NP_AMC22, Nagaiah_AMC19}.
In the study documented in \cite{Kolbe_NA14}, various time discretization methods and the adaptive finite volume approach were employed for conducting two-dimensional simulations. Furthermore, \cite{Zheng_BMB05} harnessed an adaptive Finite Element Method (FEM) to address equations related to tumor angiogenesis in a two-dimensional context. Applied adaptive FEM to investigate a simplified three-equation reaction model that computes aspects of tumor-induced angiogenesis deterministically in \cite{Peterson_IJNME07}. The study presented numerical results for geometric models across dimensions, specifically considering $d-$dimensional geometry,  where \(d=\{1,2,3\}\) and utilized explicit time-stepping schemes.
Recently, an administration technique for the bookkeeping of AMR on (hyper-)rectangular meshes is presented in \cite{Kolbe_JCAM22} and a computational  framework of space-time adaptivity for this
mode in our previous work \cite{Aswin_JCAM23}. In \cite{Aswin_JCAM23}, a computational framework based on the parallel space-adaptive techniques was presented based on the gradient type error estimator for spatial discretization. Still, the rigorous theoretical estimates of the error estimator were not presented. The previous studies investigated cancer or related models, and none of the works showed the theoretical study of a-posteriori error estimates and 
their numerical converge studies of cancer invasion models. This is the main focus of the current paper.

\par
Let $u(x,t)$, $v(x,t)$, and $w(x,t)$ be the three unknowns describing the cancer cell's density, extracellular matrix (ECM) density, and concentration of matrix-degrading enzymes (MDE), respectively. The governing equations for the dimensionless form of the cancer invasion model are as follows.
\begin{eqnarray} \label{Model:dim}
\frac{\partial u}{\partial t} & = & \nabla \cdot (d_1(u,v,w)\nabla u) - \nabla \cdot (\chi_u(v) u \nabla v) + \lambda u(1-u-v),\nonumber \\
\frac{\partial v}{\partial t} & = & \rho v(1-u-v)-\eta v w,\\
\frac{\partial w}{\partial t} & = & d_2\Delta w + \alpha u(1-w)-\beta w,\nonumber
\end{eqnarray}
where $d_1(u,v,w)$ is the density-dependent nonlinear diffusion coefficient, ${\chi_u(v)}$ is the haptotactic sensitivity function, $d_2$ is diffusion constant, and $ \lambda, \rho, \eta, \alpha, \beta$ are positive constants.\\
Assuming the interactions of the cancer cells with ECM and the degradation of ECM by MDE take place in an isolated system, zero-flux type boundary conditions are imposed on the model.
\begin{align}\label{boundary_conditions}
\xi\cdot\left(-d_1(u,v,w)\nabla u+\chi_u(v)u\nabla v\right)&=0\text{  on }\partial\Omega\times[0,T],\nonumber\\
\dfrac{\partial w}{\partial \xi}&=0\text{  on }\partial\Omega\times[0,T],
\end{align}
where $\xi$ is the outward normal vector on $\partial\Omega$.
The initial conditions for the given system of equations \eqref{Model:dim} are
\begin{align}\label{initial_conditions}
    u_0&=\left\{\begin{array}{cc}
        \exp{\left(\frac{-r^2}{\epsilon}\right)}, & r\in[0,0.25], \\
        0, & r\in(0.25,1]
    \end{array}\right.\nonumber\\
    v_0&=1-0.5u_0,\\
    w_0&=0.5u_0,\text{  in }\Omega\nonumber
\end{align}
with $\epsilon=0.01$ and $r^2=x_1^2+x_2^2+x_3^2$. \\

\par
In this current work, we focus on the
theoretical framework of a posteriori error estimates using the residual-based error estimator, and
their numerical realization is investigated. This is one of the main goals of this paper.
In section \ref{sec:weak-form}, the weak formulation of the model problem \eqref{Model:dim} and its space and temporal discretization are discussed. In section \ref{sec:error_est}, first, we derive a relationship between the error and residuals for the given problem. And then, a residual-based error estimate is proposed by finding an upper bound for the residuals. The numerical results are discussed in section \ref{sec:num-res}.

\section{Weak formulation}
\label{sec:weak-form}
Let $ \Omega \subset \mathbb{R}^3 $ be the Lipschitz continuous bounded domain. As usual, we denote $H^1(\Omega)$ as the Sobolev space of the functions in $L^2(\Omega)$ whose weak derivatives also belong to $L^2(\Omega)$. To define spaces involving time, let $ X$ be a real Banach space equipped with norm $ \|\cdot\|$ and $ T(>0) \in \mathbb{R} $. The space $ L^p(0,T;X) $ consists of all measurable functions  $ u~:~ [0,T] \rightarrow X $ with
\begin{align*}
     \|u\|_{L^p(0,T;X)} ~&:=~\left(\int_{0}^{T} \|u(t)\|^p dt \right)^{1/p} < \infty \text{  for } 1 \leq p < \infty \\
     \|u\|_{L^\infty (0,T;X)} ~&:=~\underset{0\leq t\leq T}{ess~sup}\|u(t)\| ~<\infty
\end{align*}
The space $ H^1(0,T;X) $ 
consists of all functions $ u \in L^2(0,T;X) $ such that $ u' $ exists in the weak sense and belongs to $ L^2(0,T;X) $.\par 
\begin{align*}
    \|u\|_{H^1(0,T;X)} ~:=~ \left(\int_{0}^{T} \|u(t)\|^2~+~\|u'(t)\|^2 dt\right)^{\frac{1}{2}}
\end{align*} 
Let $X(\Omega)$ be a Banach space and $X^*$ be its dual. Let $\|\cdot\|_*$ denote the dual norm of elements in $X^*$. So, from the definition of the dual norm, we have
\begin{align*}
    \|\phi\|_* = \underset{v\in X\backslash{0}}{sup}\dfrac{<\phi, v>}{\|v\|},
\end{align*}
where $<\cdot,\cdot>$ denotes the duality pairing between $X$ and $X^*$. We refer \cite{evans2010partial, adams2003pure} for more details on Sobolev spaces and Bochner spaces. To derive the weak formulation of model problem \eqref{Model:dim}, we multiply the equations \eqref{Model:dim} by test functions and integrate over $ \Omega $. After applying  integration by parts, we get the weak formulation of the model problem:

For given \(u_0,v_0,w_0 \in L^2(\Omega)\) find \(u,v,w \in L^2(0,T;H^1(\Omega)) \) with \(\displaystyle\frac{\partial u}{\partial t},\frac{\partial v}{\partial t},\frac{\partial w}{\partial t} \in   L^2(0,T;(H^1(\Omega))^*)\) such that
\begin{eqnarray} \label{w:form}
&&\left(\frac{\partial u}{\partial t}, \psi_1\right)+  (d_1(u,v,w)\nabla u, \nabla \psi_1)- (\chi_u(v) u \nabla v, \nabla \psi_1)=(\lambda u(1-u-v),\psi_1),\nonumber \\
&&\left(\frac{\partial v}{\partial t}, \psi_2\right)=(\rho v(1-u-v),\psi_2)-(\eta v w,\psi_2),\\
&&\left(\frac{\partial w}{\partial t},\psi_3 \right)+  (d_2\nabla w, \nabla \psi_3)=(\alpha u(1-w),\psi_3)-(\beta w, \psi_3),\nonumber
\end{eqnarray}
for all \( \psi_1,\psi_2,\psi_3 \in H^1(\Omega)\). Further, \((H^1(\Omega))^*\) is the dual space of \(H^1(\Omega))\).\\

To define the space discretization of weak formulation \eqref{w:form}, we adopt the conforming Galerkin method. Let $\Omega_h$ be the shape regular triangulation of $\Omega$ and $V_h$ be finite-dimensional subspace of $H^1(\Omega)$, where $h$ is the discretization parameter. The discrete problem reads as : find $u_h, v_h, w_h \in V_h$ such that for a.e. $ t\in (0,T) $ and for all $\psi_{1,h},\psi_{2,h},\psi_{3,h} \in V_h$ 
\begin{eqnarray}\label{space_dis}
    &&\left(\frac{\partial u_h}{\partial t}, \psi_{1,h}\right)+  (d_1(u_h,v_h,w_h)\nabla u_h, \nabla \psi_{1,h})- (\chi_u(v_h) u_h \nabla v_h, \nabla \psi_{1,h})\nonumber\\&&\hspace{6cm}=(\lambda u_h(1-u_h-v_h),\psi_{1,h}),\nonumber \\
&&\left(\frac{\partial v_h}{\partial t}, \psi_{2,h}\right)=(\rho v_h(1-u_h-v_h),\psi_{2,h})-(\eta v_h w_h,\psi_{2,h}),\\
&&\left(\frac{\partial w_h}{\partial t},\psi_{3,h} \right)+  (d_2\nabla w_h, \nabla \psi_{3,h})=(\alpha u_h(1-w_h),\psi_{3,h})-(\beta w_h, \psi_{3,h}).\nonumber
\end{eqnarray}

For each \( (t_i, t_j) \subset (0, T )\), consider the function space
\[
\mathcal{L}(t_i, t_j)=\Big\{z| z \in L^2(t_i, t_j ;H^1(\Omega)) \cap L^\infty(t_i, t_j;L^2(\Omega)), \frac{\partial z}{\partial t} \in L^2(t_i, t_j;(H^1(\Omega)^*) \Big\}
\]

We apply the backward Euler method for time discretization. Let \( 0=t^0<t^1 <t^2 <\cdots<t^{J}=T \) be a partition of the considered time interval \([0,T]\) with step size \(\tau^n=t^n-t^{n-1}, ~n=1,2,\ldots, J\). Thus it follows from \eqref{space_dis} 
\begin{eqnarray} \label{d:form}
&&\left(\frac{u_h^{n+1}-u^n_h}{\tau^{n+1}}, \psi_{1,h}\right)+  (d_1(u_h^{n+1},v_h^{n+1},w_h^{n+1})\nabla u_h^{n+1}, \nabla \psi_{1,h})\nonumber\\ &&\hspace{0.5in}- (\chi_u(v_h^{n+1}) u_h^{n+1} \nabla v_h^{n+1}, \nabla \psi_{1,h})-(\lambda u_h^{n+1}(1-u_h^{n+1}-v_h^{n+1}),\psi_{1,h})=0,\nonumber \\[2mm]
&&\left(\frac{v_h^{n+1}-v^n_h}{\tau^{n+1}}, \psi_{2,h}\right)-(\rho v_h^{n+1}(1-u_h^{n+1}-v_h^{n+1}),\psi_{2,h})+(\eta v_h^{n+1} w_h^{n+1},\psi_{2,h})=0,\\[2mm]
&&\left(\frac{w_h^{n+1}-w^n_h}{\tau^{n+1}},\psi_{3,h} \right)+  (d_2\nabla w_h^{n+1}, \nabla \psi_{3,h})-(\alpha u_h^{n+1}(1-w_h^{n+1}),\psi_{3,h})\nonumber\\&&\hspace{4in}+(\beta w_h^{n+1}, \psi_{3,h})=0,\nonumber
\end{eqnarray}

where $u^{n+1}_h(t), v^{n+1}_h(t), w^{n+1}_h(t) \in V_h$ for $t\in[n\tau^n, (n+1)\tau^n), 1\leq n\leq J-1$. To derive the relationship between the error and residuals, we have the following assumptions on nonlinear diffusion and sensitivity function of $\eqref{Model:dim}$.\par
For all $s_1,s_2,s_3, \Tilde{s}_1,\Tilde{s}_2,\Tilde{s}_3\in \mathbb{R}$,
\begin{align*}
     \textbf{H1}:&~ m|\xi|^2\leq \xi^Td_1(s_1,s_2,s_3)\xi\leq M|\xi|^2 ~~\text{for }M>m>0,\text{ and }\xi\in\mathbb{R}\\
     \textbf{H2}:& ~|d_1(s_1,s_2,s_3)-d_1(\Tilde{s}_1,\Tilde{s}_2,\Tilde{s}_3)|\leq d\left(|s_1-\Tilde{s}_1|+|s_2-\Tilde{s}_2|+|s_3-\Tilde{s}_3|\right),\\
     \textbf{H3}:&~\chi_u(v)\in L^{\infty}(0,T; L^{\infty}(\Omega))\\
     \textbf{H4}:&~|\chi_u(v_1)-\chi_u(v_2)|\leq L|v_1-v_2|
\end{align*}
 For the existence and uniqueness of the fully discretized system \eqref{d:form}, we refer to \cite{Aswin_JCAM23}.\\
 Here, we provide the finite element setting to derive the error estimates.
 To approximate the model problem (\ref{Model:dim}), we associate a shape regular triangulation $(\mathcal{T}_h^n)_{h>0}$ of $\Omega$ with each time step $t_n$. We denote the set of all faces of the partition $(\mathcal{T}_h^n)$ as $\mathcal{E}_h^n$. Furthermore, let $\tilde{\mathcal{T}}_h^n$ be the refinement of both triangulation $\mathcal{T}_h^n$ and $\mathcal{T}_h^{n+1}$. Also, we denote by $h_K$ the diameter of an element $K\in \tilde{\mathcal{T}}_h^n$ and $h_{K'}$ denotes the diameter of an element $K' \in \mathcal{T}_h^n$ such that $K \subseteq K'$.
\section{A posteriori error estimates}\label{sec:error_est}
This section is devoted to the study of residual operators and error estimation for the discretized problem \eqref{d:form}. In subsection 3.1, we define the residual operators and an upper bound for the error is derived in terms of the given initial data and the residuals. In subsection 3.2, we derive the upper bound for the residuals by splitting them into spatial and temporal residuals.
\subsection{Residual operators}
Suppose that \( u_h^n,v_h^n,w_h^n\) are the fully discretized solution of the problem \eqref{Model:dim} and \(\Omega^n_h\) are the different tessellations at different instants. The collection of final indices \( \{k_n \}_{n=0}^J \) is denoted as \(\bf k\). The linear interpolates \((u_h^{\bf k},v_h^{\bf k},w_h^{\bf k}) \) of continuous functions on \([0,T]\) is given by 
\[
z_h^{\bf k}=\frac{t-t_{n-1}}{\tau_n}z_{h,k_n}^n+\frac{t_n-t}{\tau_n}z^{n-1}_{h,k_{n-1}},
\]
where \(z\) denotes the variables \( u,v,w\) and \(t\in (t_{n-1},t_n], n=1,2,\ldots, J \), and $k_n$ denotes the maximum number of Newton's method iterations for each time step. Let \( \mathcal R(t)\) be the residual operator in the space \( (H^1(\Omega),H^1(\Omega), H^1(\Omega))^*=(H^1(\Omega))^*\times (H^1(\Omega))^* \times (H^1(\Omega))^*  \), and
\[
\langle \mathcal{R}(t), (\psi_1,\psi_2,\psi_3)\rangle =\langle\mathcal{R}_1(t), \psi_1 \rangle +\langle\mathcal{R}_2(t), \psi_2 \rangle +\langle\mathcal{R}_3(t), \psi_3 \rangle,
\]
where \(\langle\mathcal{R}_i(t),\psi_i \rangle, ~i=1,2,3, \) are defined as
\begin{eqnarray}\label{6}
&&\langle \mathcal{R}_1(t), \psi_1\rangle =-\int_{\Omega}\frac{\partial u^{\mathbf{k}}_h}{\partial t} \psi_1 dx-\int_{\Omega}d_1(u^{\mathbf{k}}_h,v^{\mathbf{k}}_h,w^{\mathbf{k}}_h) \nabla u^{\mathbf{k}}_h \nabla \psi_1dx\nonumber\\&&\hspace{1in}+\int_{\Omega}\chi_u(v^{\mathbf{k}}_h)u^{\mathbf{k}}_h \nabla v^{\mathbf{k}}_h \nabla \psi_1 dx+\int_{\Omega}\lambda u_h^{\mathbf{k}}(1-u^{\mathbf{k}}_h-v^{\mathbf{k}}_h)\psi_1 dx, \nonumber\\
&&\langle \mathcal{R}_2(t),\psi_2 \rangle =-\int_{\Omega}\frac{\partial v^{\mathbf{k}}_h}{\partial t }\psi_2 dx+\int_{\Omega} \rho v^{\mathbf{k}}_h(1-u^{\mathbf{k}}_h-v^{\mathbf{k}}_h)\psi_2 dx-\int_{\Omega} \eta v^{\mathbf{k}}_h w^{\mathbf{k}}_h \psi_2 dx, \\
&& \langle \mathcal{R}_3(t), \psi_3 \rangle =-\int_{\Omega} \frac{\partial w^{\mathbf{k}}_h}{\partial t} \psi_3 dx-\int_{\Omega}d_2 \nabla w^{\mathbf{k}}_h\nabla \psi_3 dx+\int_{\Omega}\alpha u^{\mathbf{k}}_h(1-w^{\mathbf{k}}_h)\psi_3 dx\nonumber\\&&\hspace{4in}-\int_{\Omega}\beta w^{\mathbf{k}}_h\psi_3 dx\nonumber
\end{eqnarray}

\begin{Lemma}\label{lem1}
Let \(\mathcal{R}_1,\mathcal{R}_2,\mathcal{R}_3\) be the residuals as defined in \eqref{6}. Then for \(t \in (0,T]\)
\begin{eqnarray}\label{7}
&&C \left(
\|u-u^{\mathbf{k}}_h \|_{\mathcal{L}(0,t)} ^2+\|v-v^{\mathbf{k}}_h \|_{L^2(0,t;H^1(\Omega))} ^2+\|w-w^{\mathbf{k}}_h \|_{\mathcal{L}(0,t)} ^2
\right)\nonumber \\&& \leq \|u_0-\Pi_h^0u_0 \|^2_{L^2(\Omega)}+\|v_0-\Pi_h^0v_0 \|^2_{L^2(\Omega)}+\|w_0-\Pi_h^0w_0 \|^2_{L^2(\Omega)}\nonumber \\&&\hspace{1in}+\|\mathcal {R}_1\|^2_{L^2(0,t;H^1(\Omega)^*)}+\|\mathcal {R}_2\|^2_{L^2(0,t;H^1(\Omega)^*)}+\|\mathcal {R}_3\|^2_{L^2(0,t;H^1(\Omega)^*)}
\end{eqnarray}
\end{Lemma}
\begin{proof}

Since $ (u,v,w) $ is the solution of \eqref{Model:dim} and by definition of $ \mathcal{R}_1 $ we have
\begin{eqnarray*}
    \langle \mathcal{R}_1(t), \psi_1\rangle =-\int_{\Omega}\frac{\partial u^{\mathbf{k}}_h}{\partial t} \psi_1 dx-\int_{\Omega}d_1(u^{\mathbf{k}}_h,v^{\mathbf{k}}_h,w^{\mathbf{k}}_h) \nabla u^{\mathbf{k}}_h \nabla \psi_1dx+\int_{\Omega}\chi_u(v^{\mathbf{k}}_h)u^{\mathbf{k}}_h \nabla v^{\mathbf{k}}_h \nabla \psi_1 dx\\ \\+\int_{\Omega}\lambda u^{\mathbf{k}}_h(1-u^{\mathbf{k}}_h-v^{\mathbf{k}}_h)\psi_1 dx
    +  \int_{\Omega}\frac{\partial u}{\partial t} \psi_1 dx+\int_{\Omega}d_1(u,v,w) \nabla u \nabla \psi_1dx\\ \\-\int_{\Omega}\chi_u(v)u \nabla v \nabla \psi_1 dx-\int_{\Omega}\lambda u(1-u-v)\psi_1 dx. \nonumber\\
\end{eqnarray*}
By taking $ \psi_1 = u-u^{\mathbf{k}}_h $ we get 
\begin{eqnarray}
&&\langle\mathcal{R}_1 , u-u^{\mathbf{k}}_h \rangle =\frac{1}{2}\frac{d}{dt}\|u-u^{\mathbf{k}}_h \|_{L^2(\Omega)}^2+\int_{\Omega}d_1(u,v,w)\nabla u \nabla(u-u^{\mathbf{k}}_h)dx\nonumber \\&&\hspace{3cm}-\int_{\Omega}d_1(u^{\mathbf{k}}_h,v^{\mathbf{k}}_h,w^{\mathbf{k}}_h)\nabla u^{\mathbf{k}}_h \nabla(u-u^{\mathbf{k}}_h)dx 
-\int_{\Omega}\chi_u(v)u\nabla v \nabla (u-u^{\mathbf{k}}_h)dx \nonumber \\&&\hspace{3cm}+\int_{\Omega}\chi_u(v^{\mathbf{k}}_h)u^{\mathbf{k}}_h\nabla v^{\mathbf{k}}_h \nabla (u-u^{\mathbf{k}}_h)dx
-\int_{\Omega}\lambda u(1-u-v)(u-u^{\mathbf{k}}_h)dx\nonumber \\&&\hspace{3cm} +\int_{\Omega}\lambda u^{\mathbf{k}}_h(1-u^{\mathbf{k}}_h-v^{\mathbf{k}}_h)(u-u^{\mathbf{k}}_h)dx 
\end{eqnarray}
Rearranging the terms, we get
\begin{eqnarray}
&&\frac{1}{2}\frac{d}{dt}\|u-u^{\mathbf{k}}_h \|_{L^2(\Omega)}^2+\int_{\Omega}d_1(u,v,w)\nabla (u-u^{\mathbf{k}}_h) \nabla(u-u^{\mathbf{k}}_h)dx = \langle\mathcal{R}_1 , u-u^{\mathbf{k}}_h \rangle  \nonumber \\ 
&&\hspace{2.5cm}+\int_{\Omega}d_1(u^{\mathbf{k}}_h,v^{\mathbf{k}}_h,w^{\mathbf{k}}_h)\nabla u^{\mathbf{k}}_h \nabla(u-u^{\mathbf{k}}_h)dx-\int_{\Omega}d_1(u,v,w)\nabla u^{\mathbf{k}}_h \nabla(u-u^{\mathbf{k}}_h)dx
\nonumber \\ 
&&\hspace{2.5cm}+\int_{\Omega}\chi_u(v)u\nabla v \nabla (u-u^{\mathbf{k}}_h)dx-\int_{\Omega}\chi_u(v)u\nabla v^{\mathbf{k}}_h \nabla (u-u^{\mathbf{k}}_h)dx \nonumber \\ 
&&\hspace{2.5cm}+\int_{\Omega}\chi_u(v)u\nabla v^{\mathbf{k}}_h \nabla (u-u^{\mathbf{k}}_h)dx-\int_{\Omega}\chi_u(v)u^{\mathbf{k}}_h\nabla v^{\mathbf{k}}_h \nabla (u-u^{\mathbf{k}}_h)\nonumber \\
&&\hspace{2.5cm}+\int_{\Omega}\chi_u(v)u^{\mathbf{k}}_h\nabla v^{\mathbf{k}}_h \nabla (u-u^{\mathbf{k}}_h)dx-\int_{\Omega}\chi_u(v^{\mathbf{k}}_h)u^{\mathbf{k}}_h\nabla v^{\mathbf{k}}_h \nabla (u-u^{\mathbf{k}}_h)dx\nonumber \\
&&\hspace{2.5cm}+\int_{\Omega}\lambda u(1-u-v)(u-u^{\mathbf{k}}_h)dx-\int_{\Omega}\lambda u^{\mathbf{k}}_h(1-u-v)(u-u^{\mathbf{k}}_h)dx\nonumber \\
&&\hspace{2.5cm}+\int_{\Omega}\lambda u^{\mathbf{k}}_h(1-u-v)(u-u^{\mathbf{k}}_h)dx-\int_{\Omega}\lambda u^{\mathbf{k}}_h(1-u^{\mathbf{k}}_h-v^{\mathbf{k}}_h)(u-u^{\mathbf{k}}_h)dx\nonumber
\end{eqnarray}

\begin{eqnarray}
&&\hspace{1cm}=\langle\mathcal{R}_1 , u-u^{\mathbf{k}}_h \rangle+\int_{\Omega}(d_1(u^{\mathbf{k}}_h,v^{\mathbf{k}}_h,w^{\mathbf{k}}_h)-d_1(u,v,w))\nabla u^{\mathbf{k}}_h \nabla(u-u^{\mathbf{k}}_h)dx\nonumber\\
&&\hspace{2.5cm}+\int_{\Omega}\chi_u(v)(u-u^{\mathbf{k}}_h)\nabla (v-v^{\mathbf{k}}_h) \nabla (u-u^{\mathbf{k}}_h)dx\nonumber\\
&&\hspace{2.5cm}+\int_{\Omega}\chi_u(v)u^{\mathbf{k}}_h\nabla (v-v^{\mathbf{k}}_h) \nabla (u-u^{\mathbf{k}}_h)dx\nonumber\\
&&\hspace{2.5cm}+\int_{\Omega}\chi_u(v)(u-u^{\mathbf{k}}_h)\nabla v^{\mathbf{k}}_h \nabla (u-u^{\mathbf{k}}_h)dx\nonumber\\
&&\hspace{2.5cm}+\int_{\Omega}(\chi_u(v)-\chi_u(v^{\mathbf{k}}_h))u^{\mathbf{k}}_h\nabla v^{\mathbf{k}}_h \nabla (u-u^{\mathbf{k}}_h)dx\nonumber\\&&\hspace{2.5cm}+\int_{\Omega}\lambda (u-u^{\mathbf{k}}_h)(1-(u-u^{\mathbf{k}}_h)-(v-v^{\mathbf{k}}_h)-u^{\mathbf{k}}_h-v^{\mathbf{k}}_h)(u-u^{\mathbf{k}}_h)dx\nonumber\\&&\hspace{2.5cm}-\int_{\Omega}\lambda (u+v-u^{\mathbf{k}}_h-v^{\mathbf{k}}_h)u^{\mathbf{k}}_h(u-u^{\mathbf{k}}_h)dx.
\end{eqnarray}
We have $ \|\nabla u\|_{L^2(\Omega)} < \infty \text{ and } \|\nabla v\|_{L^2(\Omega)} < \infty $. Using the assumptions (\textbf{H1})--(\textbf{H4}), H\"older's inequality and Minkowski's inequality, we get
\begin{eqnarray}
&&\frac{1}{2}\frac{d}{dt}\|u-u^{\mathbf{k}}_h \|_{L^2(\Omega)}^2 +m\|\nabla(u-u^{\mathbf{k}}_h)\|_{L^2(\Omega)}^2\leq \|\mathcal{R}_1\|_*\|u-u^{\mathbf{k}}_h \|_{L^2(\Omega)}\nonumber\\&&\hspace{1.5cm}+D(\|u-u^{\mathbf{k}}_h \|_{L^2(\Omega)}+\|v-v^{\mathbf{k}}_h \|_{L^2(\Omega)}+\|w-w^{\mathbf{k}}_h \|_{L^2(\Omega)})\|\nabla(u-u^{\mathbf{k}}_h)\|_{L^2(\Omega)} \nonumber\\&&\hspace{1.5cm}+K_1\|u-u^{\mathbf{k}}_h\|_{L^2(\Omega)}\|\nabla(u-u^{\mathbf{k}}_h)\|_{L^2(\Omega)}+K_2\|u^{\mathbf{k}}_h\|_{L^2(\Omega)}\|\nabla(u-u^{\mathbf{k}}_h) \|_{L^2(\Omega)} \nonumber\\&&\hspace{1.5cm}+L'\|v-v^{\mathbf{k}}_h\|_{L^2(\Omega)}\|\nabla(u-u^{\mathbf{k}}_h) \|_{L^2(\Omega)}+\lambda\|u-u^{\mathbf{k}}_h\|^2_{L^2(\Omega)} \nonumber \\
&&\hspace{1.5cm}+\lambda\|u-u^{\mathbf{k}}_h\|_{L^2(\Omega)}\left(\|u-u^{\mathbf{k}}_h\|_{L^2(\Omega)}+ \|v-v^{\mathbf{k}}_h\|_{L^2(\Omega)}\right).
\end{eqnarray}
 Adding $ m\|u-u^{\mathbf{k}}_h\|_{L^2(\Omega)}^2$ on both sides and then applying Young's inequality with $\epsilon>0$,we get 
\begin{eqnarray}
    &&\frac{1}{2}\frac{d}{dt}\|u-u^{\mathbf{k}}_h \|_{L^2(\Omega)}^2 +m\|u-u^{\mathbf{k}}_h\|_{H^1(\Omega)}^2\leq  \frac{1}{2\epsilon}\|\mathcal{R}_1\|_*^2 + \frac{\epsilon}{2}\|u-u^{\mathbf{k}}_h \|_{L^2(\Omega)}^2 + m\|u-u^{\mathbf{k}}_h\|_{L^2(\Omega)}^2\nonumber \\
    &&\hspace{1cm}+\frac{3\epsilon}{2}\|\nabla(u-u^{\mathbf{k}}_h)\|^2_{L^2(\Omega)} + \frac{D^2}{2\epsilon}\left(\|u-u^{\mathbf{k}}_h\|^2_{L^2(\Omega)}+\|v-v^{\mathbf{k}}_h\|^2_{L^2(\Omega)} +\|w-w^{\mathbf{k}}_h\|^2_{L^2(\Omega)} \right)\nonumber \\
    &&\hspace{1cm}+\frac{K_1^2}{2\epsilon}\|u-u^{\mathbf{k}}_h\|_{L^2(\Omega)}^2 + \frac{\epsilon}{2}\|\nabla(u-u^{\mathbf{k}}_h)\|_{L^2(\Omega)}^2 
    + \frac{K_2^2}{2\epsilon}\|u^{\mathbf{k}}_h\|_{L^2(\Omega)}^2+\frac{\epsilon}{2}\|\nabla(u-u^{\mathbf{k}}_h) \|_{L^2(\Omega)}^2\nonumber \\
    &&\hspace{1cm}+\frac{L'^2}{2\epsilon}\|v-v^{\mathbf{k}}_h\|_{L^2(\Omega)}^2 + \frac{\epsilon}{2}\|\nabla(u-u^{\mathbf{k}}_h)\|_{L^2(\Omega)}^2  +\lambda\|u-u^{\mathbf{k}}_h\|^2_{L^2(\Omega)}+K_3\|u-u^{\mathbf{k}}_h\|^2_{L^2(\Omega)} \nonumber\\
    &&\hspace{1cm}+\frac{K_3^2}{2\epsilon}\|v-v^{\mathbf{k}}_h\|^2_{L^2(\Omega)}+\frac{\epsilon}{2}\|u-u^{\mathbf{k}}_h\|^2_{L^2(\Omega)}
\end{eqnarray}
Simplifying further,
\begin{eqnarray}\label{9}
    &&\frac{1}{2}\frac{d}{dt}\|u-u^{\mathbf{k}}_h \|_{L^2(\Omega)}^2+(m-3\epsilon)\|u-u^{\mathbf{k}}_h\|^2_{H^1(\Omega)} \leq  \frac{1}{2\epsilon}\|\mathcal{R}_1\|_*^2 \nonumber\\
    &&\hspace{4cm}+ \left(\frac{\epsilon}{2}+m+\frac{D^2}{2\epsilon}+\frac{K_1^2}{2\epsilon}+\lambda+K_3+\frac{\epsilon}{2}\right)\|u-u^{\mathbf{k}}_h \|_{L^2(\Omega)}^2\\
    &&\hspace{4cm}+\left(\frac{D^2}{2\epsilon}+\frac{K_3^2}{2\epsilon}+\frac{L'^2}{2\epsilon}\right)\|v-v^{\mathbf{k}}_h\|^2_{L^2(\Omega)}+\frac{D^2}{2\epsilon}\|w-w^{\mathbf{k}}_h\|_{L^2(\Omega)}^2\nonumber
\end{eqnarray}

In similar kinds of calculations, we get
\begin{eqnarray}
&&\frac{1}{2}\frac{d}{dt}\|v-v^{\mathbf{k}}_h \|_{L^2(\Omega)}^2 \leq \frac{1}{2\epsilon}\|\mathcal{R}_2 \|^2_{*}+(\frac{3\epsilon}{2}+\rho)\| v-v^{\mathbf{k}}_h\|^2_{L^2(\Omega)}+\frac{\rho'^2}{2\epsilon}\| u-u^{\mathbf{k}}_h\|^2_{L^2(\Omega)}\nonumber\\&&\hspace{4cm}+ \frac{\eta'^2}{2\epsilon}\| w-w^{\mathbf{k}}_h\|^2_{L^2(\Omega)}\label{10}
\\
&&\frac{1}{2}\frac{d}{dt}\|w-w^{\mathbf{k}}_h \|_{L^2(\Omega)}^2 +d_2\|w-w^{\mathbf{k}}_h\|^2_{H^1(\Omega)}\leq \frac{1}{2\epsilon}\|\mathcal{R}_3 \|_*^2 +(\frac{3\epsilon}{2}+d_2)\| w-w^{\mathbf{k}}_h\|_{L^2(\Omega)}\nonumber\\&&\hspace{7.5cm}+\left(\frac{\alpha^2}{2\epsilon}+\frac{\alpha'^2}{2\epsilon}\right)\| u-u^{\mathbf{k}}_h\|^2_{L^2(\Omega)} \label{11}
\end{eqnarray}
Choosing $m>3\epsilon$ and adding \eqref{9}-\eqref{11}, we get
\begin{eqnarray}
&&\frac{1}{2}\frac{d}{dt}\Big[\|u-u^{\mathbf{k}}_h \|_{L^2(\Omega)}^2+  \|v-v^{\mathbf{k}}_h \|_{L^2(\Omega)}^2+ \|w-w^{\mathbf{k}}_h \|_{L^2(\Omega)}^2\Big] +(m-3\epsilon)\|u-u^{\mathbf{k}}_h\|^2_{H^1(\Omega)} \nonumber\\&&\hspace{0.5cm}+ d_2\|w-w^{\mathbf{k}}_h\|^2_{H^1(\Omega)} \leq \frac{1}{2\epsilon}\left(\|\mathcal{R}_1\|_*^2+|\mathcal{R}_2 \|^2_{*}+\|\mathcal{R}_3 \|_*^2\right)+C_1\|u-u^{\mathbf{k}}_h\|_{L^2(\Omega)}^2 \nonumber  \\
&&\hspace{5cm}+C_2\|v-v^{\mathbf{k}}_h\|_{L^2(\Omega)}^2 C_3\|w-w^{\mathbf{k}}_h\|_{L^2(\Omega)}^2\label{12}
\end{eqnarray}
Integrating over $0$ to $t$, we have
\begin{eqnarray}\label{17}
    &&\|u(t)-u^{\mathbf{k}}_h(t) \|_{L^2(\Omega)}^2+\|v(t)-v^{\mathbf{k}}_h(t) \|_{L^2(\Omega)}^2+\|w(t)-w^{\mathbf{k}}_h(t) \|_{L^2(\Omega)}^2\nonumber\\ 
    &&\hspace{0.5cm}+2(m-3\epsilon)\|u-u^{\mathbf{k}}_h\|^2_{L^2(0,t;H^1(\Omega))} +2d_2\|w-w^{\mathbf{k}}_h\|^2_{L^2(0,t;H^1(\Omega))} \nonumber\\&&\hspace{3cm}\leq \|u_0-\Pi_h^0u_0 \|_{L^2(\Omega)}^2+\|v_0-\Pi_h^0v_0 \|_{L^2(\Omega)}^2+\|w_0-\Pi_h^0w_0 \|_{L^2(\Omega)}^2\nonumber\\&&\hspace{3.5cm}+ \frac{1} {\epsilon}\left(\|\mathcal{R}_1\|_{L^2(0,t;H^1(\Omega)^*)}^2+\|\mathcal{R}_2\|_{L^2(0,t;H^1(\Omega)^*)}^2+\|\mathcal{R}_3\|_{L^2(0,t;H^1(\Omega)^*)}^2\right)\nonumber\\
    && \hspace{3.5cm} + C_1\int_0^t\|(u-u^{\mathbf{k}}_h)(s) \|_{L^2(\Omega)}^2ds+C_2\int_0^t\|(v-v^{\mathbf{k}}_h)(s)\|^2_{L^2(\Omega)}ds\nonumber\\
    &&\hspace{3.5cm}+C_3\int_0^t\|(w-w^{\mathbf{k}}_h)(s)\|_{L^2(\Omega)}^2ds
\end{eqnarray}
Using the integral form of Gronwall inequality,
\begin{eqnarray}\label{14}
&&\|u(t)-u^{\mathbf{k}}_h(t)\|^2_{L^2(\Omega)}+\|v(t)-v^{\mathbf{k}}_h(t)\|^2_{L^2(\Omega)}+\|w(t)-w^{\mathbf{k}}_h(t) \|^2_{L^2(\Omega)}\nonumber \\&&\hspace{ 0.5in} \leq e^{Ct}\Bigg( \|u_0-\Pi_h^0u_0 \|^2_{L^2(\Omega)}+\|v_0-\Pi_h^0v_0 \|^2_{L^2(\Omega)}+\|w_0-\Pi_h^0w_0 \|^2_{L^2(\Omega)}\nonumber\\
&&\hspace{0.8in}+\|\mathcal{R}_1 \|_{L^2(0,t;H^1(\Omega)^*)}^2+\|\mathcal{R}_2 \|_{L^2(0,t;H^1(\Omega)^*)}^2+\|\mathcal{R}_3 \|_{L^2(0,t;H^1(\Omega)^*)}^2\Bigg) 
\end{eqnarray}
Since $t\in(0,T]$ is arbitrary, we have
\begin{eqnarray}\label{19}
 &&\|u-u^{\mathbf{k}}_h\|^2_{L^{\infty}(0,t;L^2(\Omega))}+\|v-v^{\mathbf{k}}_h\|^2_{L^{\infty}(0,t;L^2(\Omega))}+\|w-w^{\mathbf{k}}_h \|^2_{L^{\infty}(0,t;L^2(\Omega))}\nonumber \\&&\hspace{ 0.5in} \leq e^{Ct}\Bigg( \|u_0-\Pi_h^0u_0 \|^2_{L^2(\Omega)}+\|v_0-\Pi_h^0v_0 \|^2_{L^2(\Omega)}+\|w_0-\Pi_h^0w_0 \|^2_{L^2(\Omega)}\nonumber\\
&&\hspace{0.8in}+\|\mathcal{R}_1 \|_{L^2(0,t;H^1(\Omega)^*)}^2+\|\mathcal{R}_2 \|_{L^2(0,t;H^1(\Omega)^*)}^2+\|\mathcal{R}_3 \|_{L^2(0,t;H^1(\Omega)^*)}^2\Bigg)    
\end{eqnarray}
Let $\mathcal{M}=\min\{2(m-3\epsilon),2d_2\}$. Then, from \eqref{17} we get
\begin{eqnarray}\label{20}
    &&\mathcal{M}\left(\|u-u^{\mathbf{k}}_h\|^2_{L^2(0,t;H^1(\Omega))} +\|w-w^{\mathbf{k}}_h\|^2_{L^2(0,t;H^1(\Omega))}\right) \leq \|u_0-\Pi_h^0u_0 \|_{L^2(\Omega)}^2\nonumber \\ 
    &&\hspace{2cm}+\|v_0-\Pi_h^0v_0 \|_{L^2(\Omega)}^2+\|w_0-\Pi_h^0w_0 \|_{L^2(\Omega)}^2\nonumber \\ 
    &&\hspace{2cm}+ \frac{1} {\epsilon} \left(\|\mathcal{R}_1\|_{L^2(0,t;H^1(\Omega)^*)}^2 +\|\mathcal{R}_2\|_{L^2(0,t;H^1(\Omega)^*)}^2+\|\mathcal{R}_3\|_{L^2(0,t;H^1(\Omega)^*)}^2\right)\nonumber\\
    && \hspace{2cm} + Ct\Big(\|u-u^{\mathbf{k}}_h \|_{L^{\infty}(0,t;L^2(\Omega))}^2+\|v-v^{\mathbf{k}}_h\|^2_{L^{\infty}(0,t;L^2(\Omega))}\nonumber\\
    && \hspace{8cm}+\|w-w^{\mathbf{k}}_h\|_{L^{\infty}(0,t;L^2(\Omega))}^2\Big)
\end{eqnarray}
From \eqref{6}, 
\begin{eqnarray}
&&\Big\|\frac{\partial }{\partial t}(u-u^{\mathbf{k}}_h)  \Big\|_*+\Big\|\frac{\partial }{\partial t}(v-v^{\mathbf{k}}_h)  \Big\|_*+\Big\|\frac{\partial }{\partial t}(w-w^{\mathbf{k}}_h)  \Big\|_*
 \nonumber \\ &&\hspace{1cm}\leq K \Big[\|u-u^{\mathbf{k}}_h \|_{L^2(\Omega)}+\|v-v^{\mathbf{k}}_h \|_{L^2(\Omega)}+\|w-w^{\mathbf{k}}_h \|_{L^2(\Omega)}\nonumber\\&&\hspace{2cm}+\|\mathcal{R}_1 \|_*+\|\mathcal{R}_2 \|_*+\|\mathcal{R}_3 \|_* \Big]
\end{eqnarray}
Now, we integrate over \(0\) to \(t\), 
\begin{eqnarray}\label{15}
&&\Big\|\frac{\partial }{\partial t}(u(t)-u^{\mathbf{k}}_h(t))  \Big\|_{L^2(0,t;(H^1(\Omega))^*)}+\Big\|\frac{\partial }{\partial t}(v-v^{\mathbf{k}}_h)  \Big\|_{L^2(0,t;(H^1(\Omega))^*)}\nonumber\\&&\hspace{1cm}+\Big\|\frac{\partial }{\partial t}(w-w^{\mathbf{k}}_h) \Big \|_{L^2(0,t;(H^1(\Omega))^*)} \leq K \Big[\|u-u^{\mathbf{k}}_h \|_{L^{\infty}(0,t;L^2(\Omega))}+\|v-v^{\mathbf{k}}_h \|_{L^{\infty}(0,t;L^2(\Omega))}\nonumber \\  &&\hspace{2.5in}+\|w-w^{\mathbf{k}}_h \|_{L^{\infty}(0,t;L^2(\Omega))}+\|\mathcal{R}_1 \|_{L^2(0,t;(H^1(\Omega))^*)}\nonumber \\  &&\hspace{2.5in}+\|\mathcal{R}_2 \|_{L^2(0,t;(H^1(\Omega))^*)}+\|\mathcal{R}_3 \|_{L^2(0,t;(H^1(\Omega)^*)}\Big]
\end{eqnarray}
Using \eqref{19}, \eqref{20} and \eqref{15}
\begin{eqnarray}
&&C \left(
\|u-u^{\mathbf{k}}_h \|_{\mathcal{L}(0,t)} ^2+\|v-v^{\mathbf{k}}_h \|_{L^2(0,t;H^1(\Omega))} ^2+\|w-w^{\mathbf{k}}_h \|_{\mathcal{L}(0,t)} ^2
\right)\nonumber \\&& \leq \|u_0-\Pi_h^0u_0 \|^2_{L^2(\Omega)}+\|v_0-\Pi_h^0v_0 \|^2_{L^2(\Omega)}+\|w_0-\Pi_h^0w_0 \|^2_{L^2(\Omega)}\nonumber \\&&\hspace{1in}+\|\mathcal {R}_1\|^2_{L^2(0,t;H^1(\Omega)^*)}+\|\mathcal {R}_2\|^2_{L^2(0,t;H^1(\Omega)^*)}+\|\mathcal {R}_3\|^2_{L^2(0,t;H^1(\Omega)^*)} \nonumber.
\end{eqnarray}
Hence, we proved it.
\end{proof}

\subsection{ A posteriori estimators}
We split the residual operators mainly into two operators, namely, the residual form of the space discretization \(\mathcal{R}_i^h(t)\), and temporal discretization \( \mathcal{R}_i^\tau(t)\).\\
Let $g_1(u,v)=\chi_u(v)u$ and $g_2(u,v)=\lambda u(1-u-v)$. Then for the spatial discretization,
\begin{eqnarray}
 &&\langle\mathcal{R}_1^h(t),\psi \rangle =\int_{\Omega}\frac{u^n_{h,k_n}-u^{n-1}_{h,k_{n-1}}}{\tau^n} \psi dx+\int_{\Omega}d_1(u^n_{h,k_{n}},v^n_{{h,k_{n}}},w^n_{h,k_{n}})\nabla u^n_{{h,k_{n}}}\nabla \psi\nonumber \\ &&\hspace{2cm}-\int_{\Omega}g_1(u^n_{h,k_{n}},v^n_{h,k_{n}})\nabla v^n_{h,k_{n}} \nabla \psi dx-\int_{\Omega}g_2(u^n_{h,k_{n}},v^n_{h,k_{n}}) \psi  dx.
\end{eqnarray}
For the temporal discretization, \(\mathcal{R}_1^{\tau} (t)\)
\begin{eqnarray}\label{18}
&&\langle \mathcal{R}_1^\tau (t),\psi \rangle =-\int_{\Omega}d_1(u^{\mathbf{k}}_h,v^{\mathbf{k}}_h,w^{\mathbf{k}}_h)\nabla u^{\mathbf{k}}_h\nabla \psi+\int_{\Omega}d_1(u^n_{h,k_{n}},v^n_{{h,k_{n}}},w^n_{h,k_{n}})\nabla u^n_{{h,k_{n}}}\nabla \psi dx\nonumber \\ && \hspace{1in}+\int_{\Omega}g_1(u^{\mathbf{k}}_h,v^{\mathbf{k}}_h)\nabla v^{\mathbf{k}}_h \nabla \psi-\int_{\Omega}g_1(u^n_{h,k_{n}},v^n_{h,k_{n}})\nabla v^n_{h,k_{n}} \nabla \psi\nonumber \\&&\hspace{1in}+\int_{\Omega}g_2(u^{\mathbf{k}}_h,v^{\mathbf{k}}_h) \psi  dx-\int_{\Omega}g_2(u^n_{h,k_{n}},v^n_{h,k_{n}}) \psi  dx.
\end{eqnarray}
\subsubsection{Spatial indicators}
On the interval $ (t_{n-1},t_n] $, let $ d_{1,h}\text{ and } g_{1,h} $ be the $ L^2$--projections of $ d_1\text{ and } g_1 $ onto the finite element space $ \Omega_h^n $. The element and edge residuals for $ \eqref{Model:dim}_1 $ are respectively defined as
\begin{eqnarray}
\begin{split}
	&&\rho_1^K=\frac{u_h^{n}-u^{n-1}_h}{\tau^{n}}-\nabla \cdot (d_{1,h}(u_h^{n},v_h^{n},w_h^{n}) \nabla u_h^{n})+\nabla  \cdot  (g_{1,h}(u_h^{n},v_h^{n}) \nabla v_h^{n})-g_{2}(u_h^{n},v_h^{n}),\\
	&&\rho_1^E=-\mathbb{J}_E(n_E \cdot(d_{1,h}(u_h^{n},v_h^{n},w_h^{n})\nabla u_h^n - g_{1,h}(u_h^n, v_h^n)\nabla v_h^n)) \Big|_{E \in \mathcal{E}_h^n},
 \end{split}
\end{eqnarray}
where $\mathbb{J}_E(.)$ denote the jump function across the edge $E$.
The element-wise and edge-wise data errors are, respectively
\begin{eqnarray}
&&\eta^K_1 = \nabla \cdot \Big(d_{1}(u_h^{n},v_h^{n},w_h^{n}) \nabla u_h^{n}-g_1(u_h^{n},v_h^{n}) \nabla v_h^{n}-d_{1,h}(u_h^{n},v_h^{n},w_h^{n})\nabla u_h^{n}\nonumber\\&&\hspace{7cm}+g_{1,h}(u_h^{n},v_h^{n}) \nabla v_h^{n}\Big) \\ \nonumber \\
&&\eta^E_1 = -\mathbb{J}_E(n_E \cdot(d_{1}(u_h^{n},v_h^{n},w_h^{n})\nabla u_h^n - g_{1}(u_h^n, v_h^n)\nabla v_h^n-d_{1,h}(u_h^{n},v_h^{n},w_h^{n})\nabla u_h^n \nonumber\\&&\hspace{7cm}+ g_{1,h}(u_h^n, v_h^n)\nabla v_h^n)) \Big|_{E \in \mathcal{E}_h^n}.\nonumber
\end{eqnarray}
Similarly,we define \(\rho_2^K,\rho_3^K,\rho_2^E,\rho_3^E,\eta_2^K,\eta_3^K,\eta_2^E,\eta_3^E. \)
Further, we define the spatial error indicator as \[\alpha_k^n =
\sum_{i=1}^3 \sum_{K \in \mathcal{T}_h^n} h_K^2\|\rho_i^K \|^2_{L^2(K)}+\sum_{i=1}^3 \sum_{E \in  \mathcal{E}_h^n} h_E\|\rho_i^E \|^2_{L^2(E)}
\]
and the spatial data error indicator as
\begin{equation}
\Theta_k^n =
\sum_{i=1}^3 \sum_{K \in \mathcal{T}_h^n} h_K^2\|\eta_i^K \|^2_{L^2(K)}+\sum_{i=1}^3 \sum_{E \in  \mathcal{E}_h^n} h_E\|\eta_i^E \|^2_{L^2(E)}
\label{eq:error_indicator}
\end{equation}

\subsubsection{Time indicators}
We define the indicators for temporal residuals
\begin{eqnarray}
&&P_1(t) = \|D(d_1(u^n_{h,k_n},v^n_{h,k_n},w^n_{h,k_n})\nabla u^n_{h,k_n})\|_{L^2(t_{n-1},t_n;H^1(\Omega)^*)}^2\nonumber \\
&&\hspace{1.5cm}+\|D(g_1(u^n_{h,k_n},v^n_{h,k_n})\nabla v^n_{h,k_n})\|_{L^2(t_{n-1},t_n;H^1(\Omega)^*)}^2\nonumber\\
&&\hspace{1.5cm}+\|D(g_2(u^n_{h,k_n},v^n_{h,k_n}))\|_{L^2(t_{n-1},t_n;H^1(\Omega)^*)}^2 
\end{eqnarray}
where $D$ denotes the first-order differential operator and
\begin{eqnarray}
    \kappa_k^n=\left (\|u^n_{h,k_n}-u^{n-1}_{h,k_{n-1}}\|^2+\|v^n_{h,k_n}-v^{n-1}_{h,k_{n-1}}\|^2+\|w^n_{h,k_n}-w^{n-1}_{h,k_{n-1}}\|^2\right)^2
\end{eqnarray}
Similarly, we can define the \(P_2,P_3.\) Let \(\gamma_k^n=P_1(t)+P_2(t)+P_3(t). \)
\begin{Lemma}\label{lem2}
There exists a positive constant \(C^*\) depending on the maximal ratio of the diameter of an element to the diameter of the largest ball inscribed in it and on the ratio \(\displaystyle \frac{h_{K'}}{h_K}\) such that
\begin{eqnarray}
\|\mathcal {R}_1^h\|^2_{L^2(t_{n-1},t_n;H^1(\Omega)^*))}+\|\mathcal {R}_2^h\|^2_{L^2(t_{n-1},t_n;H^1(\Omega)^*))}+\|\mathcal {R}_3^h\|^2_{L^2(t_{n-1},t_n;H^1(\Omega)^*))} \leq C^* (\alpha_k^n+\Theta_k^n),
\end{eqnarray}
in  \((t_{n-1},t_n]\) and for each \(n=1,2,3,\ldots J\).
\end{Lemma}
\begin{proof}

For \(t\in (t_{n-1},t_n]\),  $L^2$-representation of the residual yields
\begin{eqnarray}
&&\langle \mathcal{R}_1^h,\psi_1 \rangle=\sum_{K\in \tilde{\mathcal{T}}^n_h}\int_{K}\frac{u^n_{h,k_n}-u^{n-1}_{h,k_{n-1}}}{\tau^n} \psi_1 dx+\sum_{K\in \tilde{\mathcal{T}}^n_h}\int_{K}d_1(u^n_{h,k_{n}},v^n_{{h,k_{n}}},w^n_{h,k_{n}})\nabla u^n_{{h,k_{n}}}\nabla \psi_1\nonumber \\ &&\hspace{2cm}-\sum_{K\in \tilde{\mathcal{T}}^n_h}\int_{K}g_1(u^n_{h,k_{n}},v^n_{h,k_{n}})\nabla v^n_{h,k_{n}} \nabla \psi_1 dx-\sum_{K\in \tilde{\mathcal{T}}^n_h}\int_{K}g_2(u^n_{h,k_{n}},v^n_{h,k_{n}}) \psi_1  dx\nonumber\\
&&\hspace{1.35cm}=\sum_{K\in \tilde{\mathcal{T}}^n_h}\int_{K}\frac{u^n_{h,k_n}-u^{n-1}_{h,k_{n-1}}}{\tau^n} \psi_1 dx-\int_{K}\nabla\cdot(d_1(u^n_{h,k_{n}},v^n_{{h,k_{n}}},w^n_{h,k_{n}})\nabla u^n_{{h,k_{n}}})\psi_1\nonumber\\&&\hspace{2cm}+\sum_{K\in \tilde{\mathcal{T}}^n_h}\int_{\partial K}n_{K}\cdot d_1(u^n_{h,k_{n}},v^n_{{h,k_{n}}},w^n_{h,k_{n}})\nabla u^n_{{h,k_{n}}}\psi_1\nonumber\\
&&\hspace{2cm}+\sum_{K\in \tilde{\mathcal{T}}^n_h}\int_{K}\nabla\cdot\left(g_1(u^n_{h,k_{n}},v^n_{h,k_{n}})\nabla v^n_{h,k_{n}}\right)\psi_1 dx\nonumber\\
&&\hspace{2cm}-\sum_{K\in \tilde{\mathcal{T}}^n_h}\int_{\partial K}n_{K}\cdot g_1(u^n_{h,k_{n}},v^n_{h,k_{n}})\nabla v^n_{h,k_{n}}\psi_1 dx\nonumber\\
&&\hspace{2cm}-\sum_{K\in \tilde{\mathcal{T}}^n_h}\int_{K}g_2(u^n_{h,k_{n}},v^n_{h,k_{n}}) \psi_1  dx\nonumber\\
&&\hspace{1.35cm}=\sum_{K\in \tilde{\mathcal{T}}^n_h}\int_{K} \rho_1^K \psi_1 dx+\sum_{E\in\tilde{\mathcal{E}}_h^n}\int_{E}\rho_1^E\psi_1 dx\nonumber\\
&&\hspace{1.8cm}+\sum_{K\in \tilde{\mathcal{T}}^n_h}\int_{K} \eta_1^K \psi_1 dx+\sum_{E\in \tilde{\mathcal{E}}_h^n}\int_{E}\eta_1^E\psi_1 dx
\end{eqnarray}
Let $ I_h\psi_1 $ be a quasi-interpolation \cite{Verfuerth13} of $ \psi_1 $. Using the orthogonality property, we get
\begin{eqnarray}
&&|\langle \mathcal{R}_1^h,\psi_1 \rangle |\leq |\langle \mathcal{R}_1^h,I_h\psi_1 \rangle+\langle \mathcal{R}_1^h,(\psi_1-I_h\psi_1) \rangle| \nonumber\\
&& \hspace{1.8cm}\leq \sum_{K\in \tilde{\mathcal{T}}^n_h}\int_{K} \Big|\rho_1^K (\psi_1-I_h\psi_1) \Big|dx+\sum_{E\in \tilde{\mathcal{E}}_h^n}\int_{E}\Big|\rho_1^E(\psi_1-I_h\psi_1)\Big|dx\nonumber\\
&&\hspace{2.2cm}+ \sum_{K\in \tilde{\mathcal{T}}^n_h}\int_{K} \Big|\eta_1^K (\psi_1-I_h\psi_1) \Big|dx+\sum_{E\in \tilde{\mathcal{E}}_h^n}\int_{E}\Big|\eta_1^E(\psi_1-I_h\psi_1)\Big|dx  
\end{eqnarray}
By interpolation estimates for \(I_h\) and a standard trace theorem \cite{Verfuerth13}, we get
\begin{eqnarray}
&&|\langle \mathcal{R}_1^h,\psi_1 \rangle | \leq C_1\Big(\sum_{K\in \tilde{\mathcal{T}}^n_h} h_K\|\rho^K_1 \|_{L^2(K)}\|\nabla \psi_1 \|_{L^2(\omega_K)}+\sum_{E\in \tilde{\mathcal{E}}_h^n} h^\frac{1}{2}_E\|\rho^E_1 \|_{L^2(E)}\|\nabla \psi_1 \|_{L^2(\omega_E)}\nonumber\\
&&\hspace{2cm}+\sum_{K\in \tilde{\mathcal{T}}^n_h} h_K\|\eta^K_1 \|_{L^2(K)}\|\nabla \psi_1 \|_{L^2(\omega_K)}+\sum_{E\in \tilde{\mathcal{E}}_h^n} h^\frac{1}{2}_E\|\eta^E_1 \|_{L^2(E)}\|\nabla \psi_1 \|_{L^2(\omega_E)}\Big) 
\end{eqnarray}
Then by Cauchy-Schwartz inequality for sums and shape regularity of $\mathcal{T}$,
\begin{eqnarray}
&&\|\mathcal{R}_1^h \|_{*}\leq C_1 \Bigg( \sum_{K\in \tilde{\mathcal{T}}^n_h} h_K^2\|\rho^K_1 \|_{L^2(K)}^2+\sum_{E\in \tilde{\mathcal{E}}_h^n} h_E\|\rho^E_1 \|_{L^2(E)} ^2\nonumber\\&&\hspace{2.8cm}+ \sum_{K\in \tilde{\mathcal{T}}^n_h} h_K^2\|\eta^K_1 \|_{L^2(K)}^2+\sum_{E\in \tilde{\mathcal{E}}_h^n} h_E\|\eta^E_1 \|_{L^2(E)} ^2 \Bigg)^\frac{1}{2}. \label{21}
\end{eqnarray}
Similarly for $i=2,3$, we get
\begin{eqnarray}
&&\|\mathcal{R}_i^h \|_{*}\leq C_i \Big( \sum_{K\in \tilde{\mathcal{T}}^n_h} h_K^2\|\rho^k_i \|_{L^2(K)}^2+\sum_{E\in \tilde{\mathcal{E}}_h^n} h_E\|\rho^E_i \|_{L^2(E)} ^2\nonumber\\&&\hspace{2.5cm}+ \sum_{K\in \tilde{\mathcal{T}}^n_h} h_K^2\|\eta^k_i \|_{L^2(K)}^2+\sum_{E\in \tilde{\mathcal{E}}_h^n} h_E\|\eta^E_i \|_{L^2(E)} ^2\Big)^\frac{1}{2}. \label{22}
\end{eqnarray}
Here all the constants \( C_i, ~i=1,2,3\) depend on the ratio \(\frac{h_{K'}}{h_{K}} \). Squaring and adding \eqref{21} and \eqref{22}, and integrating over \((t_{n-1},t_n) \), it proves the Lemma \ref{lem2}.
\end{proof}

\begin{Lemma}\label{4}
	There exists a positive constant \(C^\dagger\)  such that
\begin{eqnarray}
\sum_{i=1}^3\|\mathcal{R}_i^\tau \|_{L^2(t_{n-1},t_n;H^1(\Omega)^*))}^2	 \leq  C^\dagger (\gamma_k^n+\kappa_k^n)
\end{eqnarray}
in the interval \((t_{n-1},t_n] \) and for each \(n=1,2,3,\ldots, J\).
\end{Lemma}
\begin{proof}
We define $G_1 : \left(L^2(0,T;H^1(\Omega))\right)^3 \rightarrow \left(L^2(0,T;H^1(\Omega)^*)\right)^3$ as,
\begin{eqnarray}
  &&\langle G_1(u,v,w), \psi_1 \rangle = \int_{\Omega}d_1(u,v,w)\nabla u\nabla \psi_1-\int_{\Omega}g_1(u,v)\nabla v \nabla \psi_1 dx\nonumber \\ &&\hspace{4cm}-\int_{\Omega}g_2(u,v) \psi_1  dx, \hspace{1cm} \forall \psi_1 \in H^1(\Omega).
\end{eqnarray}
Then, by the definition of the temporal residual \eqref{15},
\begin{align*}
   \mathcal{R}_1^\tau (t)&=G_1(u^n_{h,k_{n}},v^n_{{h,k_{n}}},w^n_{h,k_{n}})-G_1(u^{\mathbf{k}}_h,v^n_{{h,\tau}},w^{\mathbf{k}}_h) \\
    &= \int_0^1 DG_1(u^{\mathbf{k}}_h+s(u^n_{h,k_{n}}-u^{\mathbf{k}}_h),v^{\mathbf{k}}_h+s(v^n_{{h,k_{n}}}-v^{\mathbf{k}}_h),w^{\mathbf{k}}_h+s(w^n_{{h,k_{n}}}-w^{\mathbf{k}}_h))\cdot \\ 
    & \hspace{1.2cm}(u^n_{h,k_{n}}-u^{\mathbf{k}}_h,v^n_{{h,k_{n}}}-v^{\mathbf{k}}_h,w^n_{{h,k_{n}}}-w^{\mathbf{k}}_h)ds\\
    &= DG_1(u^n_{h,k_{n}},v^n_{{h,k_{n}}},w^n_{h,k_{n}})(u^n_{h,k_{n}}-u^{\mathbf{k}}_h,v^n_{{h,k_{n}}}-v^{\mathbf{k}}_h,w^n_{{h,k_{n}}}-w^{\mathbf{k}}_h)\\
    & \hspace{0.3cm}+\int_0^1 [(DG_1(u^{\mathbf{k}}_h+s(u^n_{h,k_{n}}-u^{\mathbf{k}}_h),v^{\mathbf{k}}_h+s(v^n_{{h,k_{n}}}-v^{\mathbf{k}}_h),w^{\mathbf{k}}_h+s(w^n_{{h,k_{n}}}-w^{\mathbf{k}}_h))\\
    & \hspace{1.6cm}-DG_1(u^n_{h,k_{n}},v^n_{{h,k_{n}}},w^n_{h,k_{n}}))\cdot(u^n_{h,k_{n}}-u^{\mathbf{k}}_h,v^n_{{h,k_{n}}}-v^{\mathbf{k}}_h,w^n_{{h,k_{n}}}-w^{\mathbf{k}}_h)]ds 
    \\&= \mathcal{R}_{1,1}^\tau (t) + \mathcal{R}_{1,2}^\tau (t)
\end{align*}
We define 
\begin{eqnarray}
    &&\langle r^n, \psi_1 \rangle =\int_{\Omega}\Big[d_1^u(u^{n}_{h,k_{n}},v^{n}_{{h,k_{n}}},w^{n}_{h,k_{n}})\nabla u^{n}_{h,k_{n}} \nabla \psi_1(u^{n}_{h,k_{n}}-u^{n-1}_{h,k_{n-1}})\nonumber \\
    &&\hspace{2cm}+ d_1^v(u^n_{h,k_{n}},v^n_{{h,k_{n}}},w^n_{h,k_{n}})\nabla u^n_{{h,k_{n}}}\nabla \psi_1(v^{n}_{h,k_{n}}-v^{n-1}_{h,k_{n-1}})\nonumber \\
    && \hspace{2cm} +d_1^w(u^n_{h,k_{n}},v^n_{{h,k_{n}}},w^n_{h,k_{n}})\nabla u^n_{{h,k_{n}}}\nabla \psi_1(w^{n}_{h,k_{n}}-w^{n-1}_{h,k_{n-1}}) \nonumber \\ 
    &&\hspace{2cm}+d_1(u^{n}_{h,k_{n}},v^{n}_{{h,k_{n}}},w^{n}_{h,k_{n}})\nabla (u^{n}_{h,k_{n}}-u^{n-1}_{h,k_{n-1}}) \nabla \psi_1\nonumber \\
    && \hspace{2cm}-g_1^u(u^{n}_{h,k_{n}},v^{n}_{h,k_{n}})\nabla v^{n}_{h,k_{n}} \nabla \psi_1(u^{n}_{h,k_{n}}-u^{n-1}_{h,k_{n-1}})\nonumber \\&&\hspace{2cm}-g_1^v(u^n_{h,k_{n}},v^n_{h,k_{n}})\nabla v^n_{h,k_{n}} \nabla \psi_1(v^{n}_{h,k_{n}}-v^{n-1}_{h,k_{n-1}})\nonumber \\
    && \hspace{2cm}-g_1(u^n_{h,k_{n}},v^n_{h,k_{n}})\nabla (v^n_{h,k_{n}}-v^{n-1}_{h,k_{n-1}}) \nabla \psi_1\nonumber \\&& \hspace{2cm} -g_2^u(u^{n}_{h,k_{n}},v^{n}_{h,k_{n}}) \psi_1(u^{n}_{h,k_{n}}-u^{n-1}_{h,k_{n-1}}) \nonumber \\
    && \hspace{2cm}-g_2^v(u^n_{h,k_{n}},v^n_{h,k_{n}}) \psi_1(v^{n}_{h,k_{n}}-v^{n-1}_{h,k_{n-1}}) \Big]dx.
\end{eqnarray}
Then, on $ (t_{n-1},t_n] $ we have
\begin{eqnarray}
    \langle \mathcal{R}_{1,1}^{\tau}(t) , \psi_1 \rangle = \left( 1-\dfrac{t-t_{n-1}}{\tau_n}\right)\langle r^n, \psi_1 \rangle  
\end{eqnarray}
By the Lemma 6.47 in \cite{Verfuerth13}, we get 
\begin{eqnarray}
    &&\|\mathcal{R}_{1,1}^\tau  \|^2_{L^2(t_{n-1},t_n;H^1(\Omega)^*)} \leq \frac{\tau_n}{3}(\|r^n\|^2)\nonumber\\
    &&\hspace{3.4cm}=\frac{\tau_n}{3} P_1(t)
\end{eqnarray}
Now, with conditions \textbf{H1}--\textbf{H4} and assuming that $ DG_1 $ is Locally Lipschitz continuous at the solution of \eqref{Model:dim}, we get
\begin{eqnarray}
   && \|\mathcal{R}_{1,2}^\tau (t)\|_{L^2(t_{n-1},t_n;H^1(\Omega)^*)} \leq L\left(\|u^{\mathbf{k}}_h-u^n_{h,k_{n}}\|^2+\|v^{\mathbf{k}}_h-v^n_{h,k_{n}}\|^2+\|w^{\mathbf{k}}_h-w^n_{h,k_{n}}\|^2\right)\nonumber\\
   &&\hspace{2.5cm}\leq L\tau_n\left(\|u^n_{h,k_n}-u^{n-1}_{h,k_{n-1}}\|^2+\|v^n_{h,k_n}-v^{n-1}_{h,k_{n-1}}\|^2+\|w^n_{h,k_n}-w^{n-1}_{h,k_{n-1}}\|^2\right)\nonumber\\
   &&\|\mathcal{R}_{1,2}^\tau (t)\|^2_{L^2(t_{n-1},t_n;H^1(\Omega)^*)}\leq L^* \kappa_k^n
\end{eqnarray}
Similarly, we do the calculations for \(\mathcal{R}_2^{\tau} \) and \(\mathcal{R}_3^{\tau}\) and adding the results, we get
\begin{eqnarray}
	\sum_{i=1}^3\|\mathcal{R}_i^\tau \|^2_{L^2(t_{n-1},t_n;H^1(\Omega)^*)}	 \leq  C^\dagger (\gamma_k^n+\kappa_k^n)
\end{eqnarray}
\end{proof}
\begin{Th}
For the solution given by \((u_{h,k_n}^n,v_{h,k_n}^n,w_{h,k_n}^n) \) , \( n=1,2,\ldots, N\), \(k=1,2,3\ldots,k_n\) , it holds
\begin{eqnarray}
 &&\| u-u_{h,\tau}^k\|_{\mathcal{L}(0,t_n)}^2+\| v-v_{h,\tau}^k\|_{L^2(0,t_n;H^1(\Omega))}^2+\| w-w_{h,\tau}^k\|_{\mathcal{L}(0,t_n)}^2 \nonumber\\&&\hspace{1cm}\leq C \left( \|u_0-u_{h,0}\|^2_{L^2(\Omega)} +\|v_0-v_{h,0}\|^2_{L^2(\Omega)}\right. + \|w_0-w_{h,0}\|^2_{L^2(\Omega)} \nonumber \\ &&\hspace{4cm}\left. +\alpha_K+\Theta_k+\gamma_k+\kappa_k  \right),
 \end{eqnarray} 
where \(C\) is the positive constant which depends on the maximal ratio of the diameter of an element to the diameter of the largest  ball inscribed  in it and \( \frac{h_{K'}}{h_K}. \)
\end{Th}
\begin{proof}
 By the definition of the residuals,
	\begin{equation*}
     \mathcal{R}_i(t)=\mathcal{R}_i^h+\mathcal{R}_i^\tau(t) \mbox{ for }i=1,2,3.
	\end{equation*}
	The proof follows from Lemma \ref{lem1}-\ref{4}. 
\end{proof}

\section{Numerical Experiments}
\label{sec:num-res}
This section presents a series of numerical experiments with the
adaptive mesh refinement strategy based on the residual error estimators that were analyzed in section \ref{sec:error_est}. Then, we compute its convergence order when $h$ diminishes. We used piecewise linear finite elements
(i.e., polynomials of degree one) in all the experiments. The numerical simulations are performed in a 3D spatial domain that has the size of 
$\Omega = [0,1] \times [0,1]\times [0,1]$.
The presented numerical results are done on our local desktop, which consists of an Intel Core-i9 processor clocked at 2.80 GHz and equipped with 32 GB RAM.
In this regard, we consider our first test example based on \cite{ganesan2017biophysical}, which consists of the following model parameters (set-1):
\begin{equation}\label{problem1}
    c = d_1 = 0.0001,~ d_2 = 0.0005, ~\chi = 0.005, ~ \lambda = 0.75, ~ \eta = 10,~ \rho = 1.5, ~ \beta = 0.1,~ \alpha = 0.25. \nonumber
\end{equation}
In our computational domain, we choose the origin as the starting point for solid tumor formation and expand its size based on the time period. The initial and boundary conditions (Eqs. \eqref{boundary_conditions}-\eqref{initial_conditions})  are applied to a coarse grid with $20\times 20\times 20$ mesh elements which consist of 48000 tetrahedrons and 9261 nodes. The Adaptive Mesh Refinement (AMR) method relies on a hierarchical structure of nested mesh levels and is a recognized strategy for effectively managing mesh resolutions to ensure reliability. Its key function involves post-processing a computed numerical solution at a particular time step within a computational domain to determine areas where the current computational mesh needs refinement or coarsening. This results in the creation of a dynamic mesh that adeptly captures the numerical solution with both efficiency and precision. In our simulations, we employ the AMR approach guided by the error indicator outlined in Eq.~\eqref{eq:error_indicator}, and the spatial relative tolerance is set to $TOL_x = 0.001$, see \cite{Aswin_JCAM23} for the detailed algorithm of the spatial grid adaptivity.

First, we illustrate the $L_2$-norm of the error between uniform and adaptive refinements in Figure~\ref{fig:uniform-adapt}.  We allow the minimal size of an element in the adaptive grid refinement up to 0.0108253. The coarse grid is refined uniformly seven times to get the model problem's reference solution due to the missing exact solution. In our computations, this fine mesh comprises 6,144,000
tetrahedral elements and 1,043,441 nodes.
As expected, the convergence rate is far better in the case of adaptive mesh refinement than with uniform refinement w.r.t the number of degrees of freedom.

\begin{figure}
\begin{center}
 \includegraphics[scale=0.25]{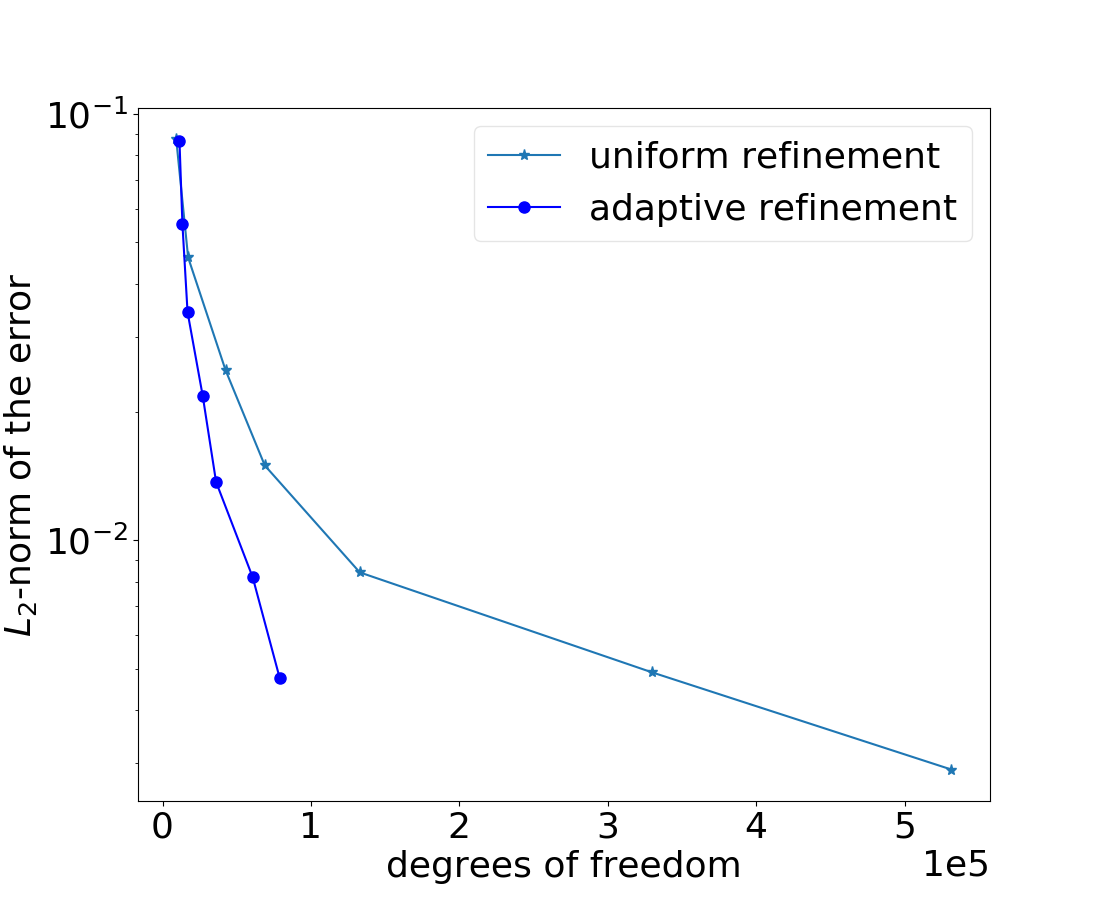}~
\end{center}
 \caption{$L_2$-norm of the error(log scale) w.r.t the degrees of freedom for the uniform and adaptive mesh refinement strategies using the parameter set 1.}
 \label{fig:uniform-adapt}
\end{figure}

The solution  of cancer density and their corresponding computational mesh obtained using adaptive grid refinement is depicted in Figure~\ref{fig:surf-mesh-sol}. We observe that the  error is concentrated at the solution boundaries in the computational domain's bottom corners, and the error indicator suggests keeping more mesh points in this region. More information on the numerical solutions to the current problem can be referred to \cite{Aswin_JCAM23}.

\begin{figure}[!h]
\centering
\includegraphics[width=0.31\textwidth] {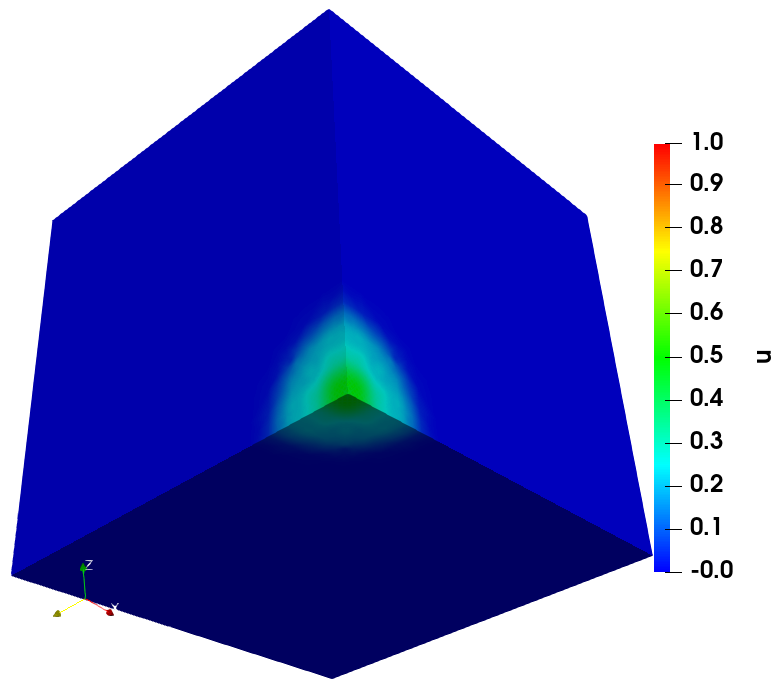}
\includegraphics[width=0.31\textwidth] {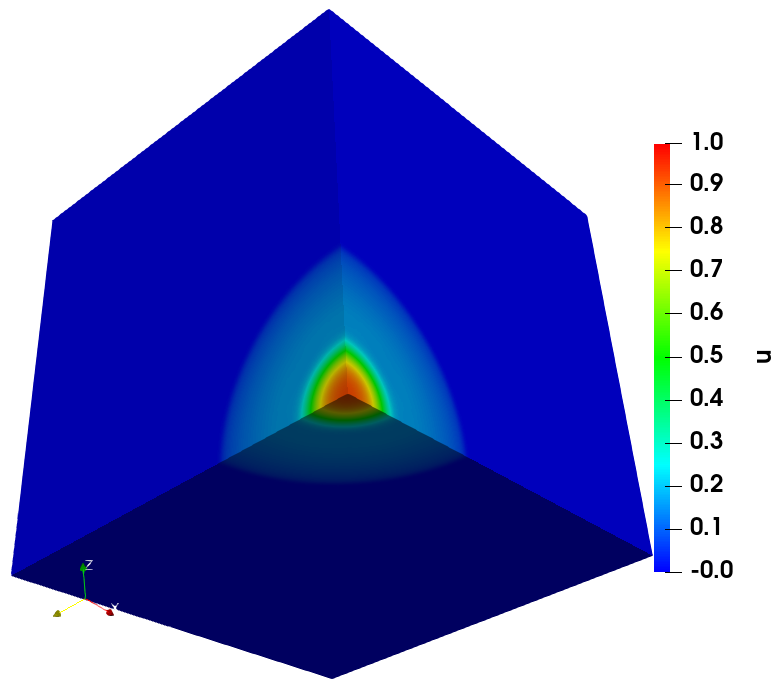}
\includegraphics[width=0.31\textwidth] {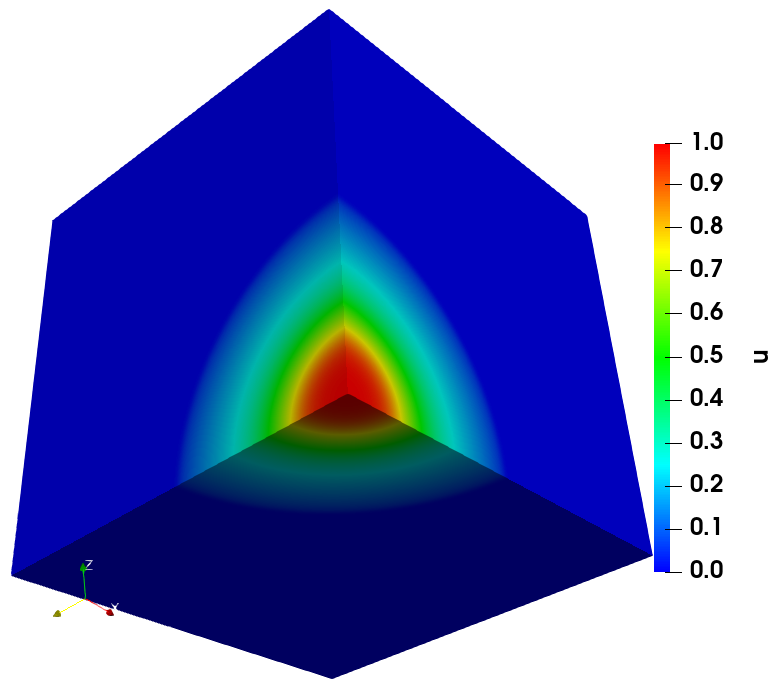}\hspace*{0.2em}

\includegraphics[width=0.27\textwidth] {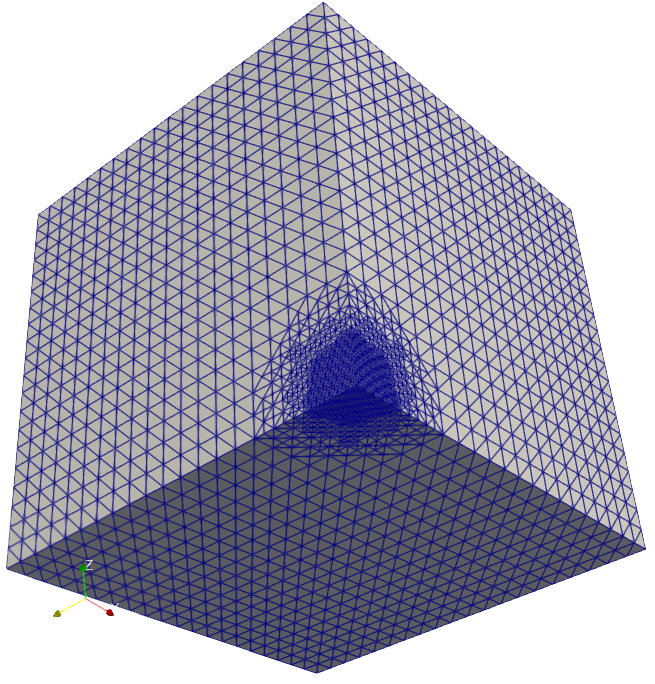}\hspace*{1.2em}
\includegraphics[width=0.27\textwidth] {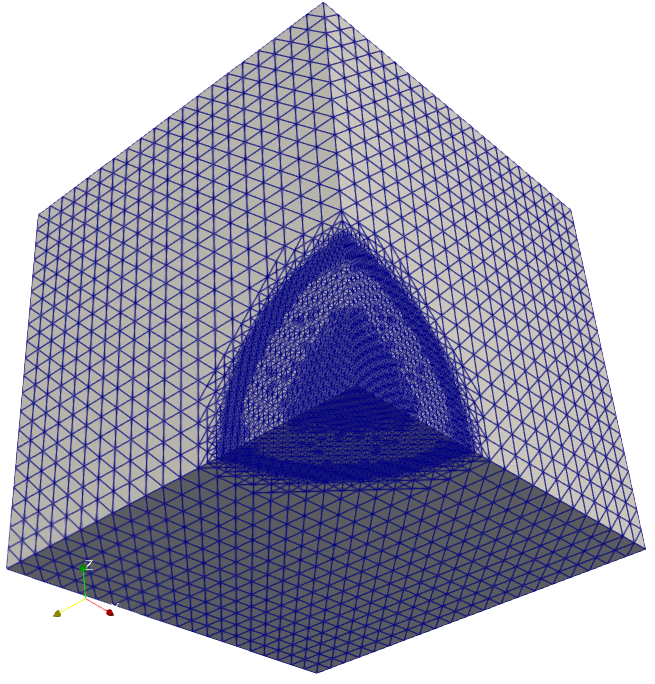}\hspace*{1.2em}
\includegraphics[width=0.27\textwidth] {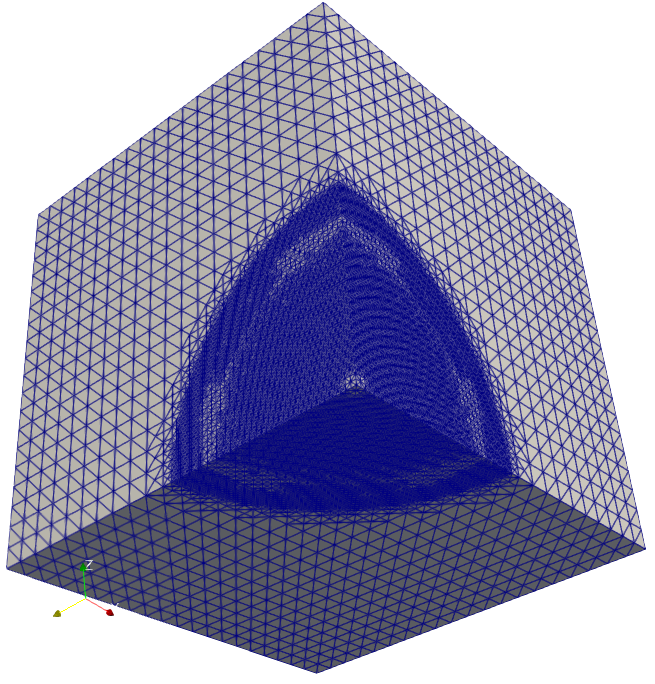}
\caption{The evolution of cancer density (at top row) and the corresponding computational mesh at different times $t = 1, 5,$ and $10$.}
\label{fig:surf-mesh-sol}
\end{figure}

\begin{table}
\begin{center}
 \begin{tabular}{|p{2.0cm}|p{2.0cm}|p{2.0cm}|} \hline 
  No. of mesh points & $L_2$ error & CPU Time (seconds) \\  \hline
11,439 &  0.0863209  & 593.5\\
13,413 &  0.0552073  & 813.7\\
16,874 &  0.0343122  & 941.5 \\
27,104 &  0.0218120  & 1302.2\\
36,107 &  0.0137021  & 1508.7 \\
60,770 &  0.0081902  & 3010.6 \\
78,945 &  0.0047541  & 3977.5 \\ \hline
 \end{tabular}
\end{center}

 \caption{$L_2$ error estimates and the corresponding CPU times at different adaptive mesh levels for the first parameter set.}
\label{table:adapt-information-1}
\end{table}

The computational times of the above simulation are given in Table~\ref{table:adapt-information-1}.  
By obtaining a similar accuracy, we remark that the adaptive simulation is 4.91 times faster than the uniform grid refinement simulation. We can observe that there is a significant improvement in the CPU times by using the spatial grid adaptivity simulations. We can observe that this test case's experimental order of convergence reaches 0.78.

The performance of the residual error estimator with a second set of model parameters
is demonstrated in the following. We use the same computational setup as the previous test case except for the following essential parameters (set-2).
\begin{equation}\label{problem2}
    \rho = 0.0, ~ \alpha = 0.1, ~ \beta = 0.0, ~ \chi = 0.00005.
\end{equation}

To find the $L_2$ error in the adaptive simulations, we used the solution of uniform refinement 
of the coarse grid with grid level 7.  Also, in this test case, we can observe that the convergence rate
of the solution is better for adaptive simulations, and note that adaptive simulation is 5.28 times faster than the static finer uniform grid simulation. 
\begin{table}
\begin{center}
 \begin{tabular}{|p{2.0cm}|p{2.0cm}|p{2.0cm}|} \hline 
  No. of mesh points & $L_2$ error & CPU Time (seconds) \\  \hline
9,852  &  0.0361524  &   449.6  \\
10,473 &  0.0197031  &   511.7  \\
11,569 &  0.0128632  &   563.3  \\
14,910 &  0.0084189  &   869.1  \\
18,331 &  0.0056551  &  1058.6  \\
26,616 &  0.0032899  &  1316.0  \\
51,843 &  0.0017817  &  2464.5  \\ \hline
 \end{tabular}
\end{center}

 \caption{$L_2$ error estimates and the corresponding CPU times at different adaptive mesh levels for the second parameter set.}
\label{table:adapt-information-2}
\end{table}

The computational times for the second test case simulation are given in Table~\ref{table:adapt-information-2}.  In this test also, there is a significant improvement in the CPU times by using the spatial grid adaptivity simulations. For this test case, the experimental order of convergence reaches 0.88.

\section{Conclusions}
This paper analyzes residual-based a posteriori error estimates of a multi-scale cancer invasion model consisting of nonlinear reaction terms and sensitivity functions. 
We have derived a residual-based a posteriori error estimator for the coupled system and shown that it is reliable. We obtained an upper bound for the discretization error using the residual-based error estimator.
The numerical results were demonstrated for two different set of parameters where the first set of parameters creates a stiffer system than the second set.
In the case of using the first set of parameters, we obtained the experimental order of convergence is 
0.78, and using the second set of parameters, it is 0.88.  
These results allow us to expect some improved computational and theoretical estimates for 
a multi-scale cancer invasion model in the future.

\subsection*{Acknowledgment}
We gratefully acknowledge the financial support under R\&D projects leading to HPC Applications (DST/NSM/R\&D\_HPC\_Applications/2021/03.28), National Supercomputing Mission, India, and the support for high-performance computing time at Param Yukti, JNCASR, India. Also, we greatly acknowledge the support for high-performance computing time at the Padmanabha cluster, IISER Thiruvananthapuram, India.

\subsection*{Disclosure/Conflict-of-Interest Statement}
The authors declare that the research was conducted in the absence of any commercial or financial relationships that could be construed as a potential conflict of interest.

\bibliographystyle{abbrv}
\bibliography{ref}

\end{document}